\documentclass[12pt]{article}

\usepackage{amssymb,amsmath,amsfonts,amsthm}
\newcounter{parag}

\newcommand\la{\langle}
\newcommand\ra{\rangle}
\newtheorem{lem}{Lemma}
\newtheorem{theorem}{Theorem}
\newtheorem{defi}{Definition}

\newtheorem{cor}{Corollary}
\newtheorem{rem}{Remark}

\title{On groups whose conjugacy class sizes are not divisible by each other}
\author{ Yang Nanying \footnote{School of Science, Jiangnan University, Wuxi, 214122, China},
	Ilya Gorshkov\footnote{Sobolev Institute of Mathematics, Novosibirsk, 630090, Russia} \footnote{Siberian Federal University, Krasnoyarsk, 660041, Russia}}

\date{\vspace{-30px}}

\begin{document}
	\maketitle
	
	
		%
		%
	
	{\it
		Abstract: Let $G$ be a finite group and $N(G)$ be the set of its conjugacy class sizes excluding~$1$.
		Let us define a directed graph $\Gamma(G)$, the set of vertices of this graph is $N(G)$ and the vertices $x$ and $y$ are connected by a directed edge from $x$ to $y$ if $x$ divides $y$ and $N(G)$ does not contain a number $z$ different from $x$ and $y$ such that $x$ divides $z$ and $z$ divides $y$. We will call the graph $\Gamma(G)$ the conjugate graph of the group $G$. In this work, we will study finite groups whose conjugate graph is a set of points.
		
		\smallskip
		
		 \smallskip
		
		
	}
	\section*{Introduction}
	Over the years, considerable work has been done to establish relations between
	the structure of a finite group and its set of sizes of conjugacy classes.
	
	We say that a group $G$ has conjugate rank $n$ (shortly rank or $rank(G)$) if $|N(G)|=n$.
	Noboru Ito laid the foundation for the study of $F$-groups in his famous paper \cite{Ito}. A finite group $G$ is $F$-group if $x, y \in G \setminus Z(G)$ and $ C_G(x)\leq C_G(y)$ implies that
	$C_G(x)=C_G(y)$. An important subclass of $F$-groups is the class of rank $1$ groups, i.e. $I$-group is a group whose set of conjugacy classes sizes is $\{1,n\}$. Ito proved that rank $1$ groups are nilpotent, in particular $n$ is a prime power. Later, Kenta Ishikawa showed \cite{Ish} that rank $1$ groups are of class at most $3$. In \cite{TRM} rank $1$ groups when $p\neq 2$ were described.
	
	Johen Rebmann \cite{Reb} proved a classification theorem describing $F$-groups. He determined their structure, up to
	$F$-groups which are central extensions of groups of prime-power order.
	
	One more subclass of $F$-groups is the class of $CA$-groups. Finite group $G$ is a $CA$-group if all centralizers of noncentral elements are abelian. The $CA$-groups were investigated by Roland Schmidt \cite{Shmidt}. He determined their structure up to $CA$-groups which are central extensions of groups of prime-power order. 
	
	Silvio Dolfi, Marcel Herzog and Enrico Jabara \cite{DHJ} studied $CH$ groups, consisting of finite groups in which noncentral commuting elements have centralizers of the same order. In particular, the following inclusion was proved 
	
	$$CA\subset CH\subset F.$$
	
	Given $\Theta\subseteq \mathbb{N}, |\Theta|<\infty$, define the directed graph $\Gamma(\Theta)$, with the vertex set $\Theta$ and edges $\overrightarrow{ab}$ whenever $a$ divides $b$ and $\Theta$ does not contain a number $c$ such that $a$ divides $c$ and $c$ divides $b$. In \cite{Gorthom1}, the conjugate graph $\Gamma(G)$ was defined for a finite group $G$. Set $\Gamma(G)=\Gamma(N(G)\setminus\{1\})$. We will say that $G$ is an $SP$-group if $\Gamma(G)$ does not contain edges. Note that the definition of $SP$-groups differs significantly from the definition of $CA$- and $CH$-groups. In this case, only the arithmetic properties of the group are used.
	
	The main goal of this manuscript is to describe groups with the $SP$ property. 
	In particular, we will prove the inclusion $SP\subset CH$. 
	
	\begin{theorem}\label{primeiro}
		$SP\subset CH.$
	\end{theorem}	
	
	Using this result and classification of $CH$-groups we obtain a classification of $SP$-groups.
	
	\begin{theorem}\label{segundo}
		A group $G$ is $SP$-group if and only if it is of one of the following types.
		\begin{enumerate}
			\item[(I)]  $G=T\times P$ where $T$ is abelian and $P$ is a $p$-group for some prime $p$, $rank(P)=1$.
			
			\item[(II)] $G/Z$ is a Frobenius group with Frobenius kernel $K/Z$ and Frobenius
			complement $L/Z$, where $K$ and $L$ are abelian.
			
			\item[(III)] $G/Z$ is a Frobenius group with Frobenius kernel $K/Z$ and Frobenius
			complement $L/Z$, such that $K=PZ$, where $rank(P)=1$ and $P$ is a normal Sylow $p$-subgroup of
			$G$ for some $p\in\pi(G)$, $Z(P)=Z\cap P$ and $L$ is abelian.
			
			\item[(IV)] $G/Z\simeq PSL(2, p^n)$ or $PGL(2, p^n)$ and $G'\simeq SL(2, p^n)$, where $p$ is a prime and $p^n>3$.
			
			\item[(V)] $G/Z\simeq PSL(2, 9)$ or $PGL(2, 9)$ and $G'$ is isomorphic to the Schur cover of
			$PSL(2, 9)$.
			
		\end{enumerate}
		
	\end{theorem}
	
	\begin{cor}\label{Cor1}
		If $\Gamma(G)$ is an edgeless graph with two vertices, then $G/Z$ is a solvable Frobenius group.
	\end{cor}
	
	In \cite{DJ09}, Dolfi and Jabara studied groups of rank $2$ and, in particular, from their description one can also obtain Corollary \ref{Cor1}.
	
	\begin{cor}
		The graph $\Gamma(G)$ of an $SP$-group $G$ contains at most 3 vertices.
	\end{cor}

	\section{Preliminaries}
	
	Let $G$ be a group and take $x\in G$. We denote by $x^G$ the conjugacy class of $G$ containing $x$ and $C_H(x)$ is the centralizer of $x$ in the subgroup $H$. If $N$ is a subgroup of $G$, then $Ind(N,x)=|N|/|C_N(x)|$. Note that $Ind(G,x)=|x^G|$.
	
	\begin{lem}[{\rm \cite[Lemma 1]{Camin}}]\label{Complement}If, for some prime $p$, every $p'$-element of a group $G$ has index prime to $p$, then the Sylow $p$-subgroup of $G$ is a direct factor of $G$.
	\end{lem}
	\begin{lem}[{\rm \cite[Lemma 1.4]{GorA2}}]\label{factorKh}
		Let $G$ be a finite group, $K\unlhd G$ and $\overline{G}= G/K$. Take $x\in G$ and $\overline{x}=xK\in G/K$.
		Then the following conditions hold
		
		(i) $|x^K|$ and $|\overline{x}^{\overline{G}}|$ divide $|x^G|$.
		
		(ii) If $L$ and $M$ are consequent members of a composition series of $G$, $L<M$, $S=M/L$, $x\in M$  and
		$\widetilde{x}=xL$ is an image of $x$, then $|\widetilde{x}^S|$ divides $|x^G|$.
		
		(iii) If $y\in G, xy=yx$, and $(|x|,|y|)=1$, then $C_G(xy)=C_G(x)\cap C_G(y)$.
		
		(iv) If $(|x|, |K|) = 1$, then $C_{\overline{G}}(\overline{x}) = C_G(x)K/K$.
		
		(v) $\overline{C_G(x)}\leq C_{\overline{G}}(\overline{x})$.
	\end{lem}	
	
	The following two lemmas are simple exercises.
	\begin{lem}\label{ab}
		Let $P$ be a $p$-group. Then $P/Z(P)$ is not a cyclic group.
	\end{lem}
	
	\begin{lem}\label{cenfak}
		Let $G$ be a finite group, $K\unlhd G, x\in G$. Then $C_G(x)K/K\leq C_{G/K}(xK)$.
	\end{lem}
	
	\begin{defi}
		A finite group $G$ is called a $CH$-group if for every $x, y\in G\setminus Z(G), xy=yx$ implies that $|C_G(x)|=|C_G(y)|$.
	\end{defi}
	
	\begin{lem}\cite[Theorem 4.2]{DJ09}\label{DHJ}
		Let $G$ be a nonabelian group and write $Z=Z(G)$. Then $G$ is a $CH$-group
		if and only if it is of one of the following types.
		\begin{enumerate}
			\item[(I)]  $G$ is nonabelian and has an abelian normal subgroup of prime index.
			
			\item[(II)] $G/Z$ is a Frobenius group with Frobenius kernel $K/Z$ and Frobenius
			complement $L/Z$, where $K$ and $L$ are abelian.
			
			\item[(III)] $G/Z$ is a Frobenius group with Frobenius kernel $K/Z$ and Frobenius
			complement $L/Z$, such that $K=PZ$, where $P$ is a normal Sylow $p$-subgroup of
			$G$ for some $p\in\pi(G)$, $P$ is a $CH$-group, $Z(P)=Z\cap P$ and $L\simeq HZ$, where $H$
			is an abelian $p'$-subgroup of $G$.
			
			\item[(IV)] $G/Z\simeq S_4$ and if $V/Z$ is the Klein four group in $G/Z$, then $V$ is nonabelian.
			
			\item[(V)] $G\simeq P\times A$, where $P$ is a nonabelian $CH$-group of prime-power order and $A$ is abelian.
			
			\item[(VI)] $G/Z\simeq PSL(2, p^n)$ or $PGL(2, p^n)$ and $G'\simeq SL(2, p^n)$, where $p$ is a prime and $p^n>3$.
			
			\item[(VII)] $G/Z\simeq PSL(2, 9)$ or $PGL(2, 9)$ and $G'$ is isomorphic to the Schur cover of
			$PSL(2, 9)$.
			
		\end{enumerate}
	\end{lem}
	
	\begin{defi}
		We will say that $G$ lies in the class $F(i)$, where $i\in\{ I, II, .., VII\}$ and write $G\in CH(i)$ if $G$ is a $CH$-group and item $i$ from Lemma \ref{DJ} is satisfied.
	\end{defi}

	\begin{defi}
		A finite group $G$ is called an $SP$-group if for every $x, y\in G\setminus Z(G)$, $|C_G(x)|$ divides $|C_G(y)|$ implies that $|C_G(x)|=|C_G(y)|$.
	\end{defi}
	\begin{rem}
		A set of integers greater than $1$ is primitive if no element of the set divides
		another. The definition of $SP$-groups is equivalent to a requirement that $N(G)$ is a primitive set. 
		
	\end{rem}
	
	\begin{lem}\label{Gor_5.2.3}{\normalfont\cite[Theorem~5.2.3]{Gore}}
		Let $A$ be a $p'$-group of automorphisms of abelian $p$-group $P$. Then $P=C_P(A)\times [P,A]$.
	\end{lem}
	
	\section{Proof of Theorem 1}
	
	Let $G\in SP$, $Z=Z(G)$, and $x,y\in G\setminus Z$ be commuting elements. Let us show that $|C_G(x)|=|C_G(y)|$. We have $x=x_1,x_2,..x_n$, where $x_i$ are elements of primary and coprime order. Note that $C_G(x_ix_j)=C_G(x_i)\cap C_G(x_j)$. Thus, among the elements $x_1,..,x_n$ there is an element $z$ that does not lie in $Z(G)$. For any $a\in C_G(z)$ such that $(a,z)=1$, we have $C_G(a)=C_G(z)$ or $C_G(a)=G$. Similarly, in $C_G(y)$ there is an element $t$ of primary order such that $C_G(t)=C_G(y)$ and for any $b\in C_G(t)$ such that $(t,b) =1$, we have $C_G(b)=C_G(t)$ or $C_G(b)=G$. Since the elements $x$ and $y$ commute, it follows that the elements $z$ and $t$ commute. If $t$ and $z$ have coprime order, then $C_G(tz)=C_G(t)\cap C_G(z)$ and therefore $C_G(t)=C_G(tz)=C_G(z)$ and the statement has been proven.
	
	Let's assume that $z$ and $t$ are $p$-elements, where $p$ is a prime number. If in the centralizer of $z$ or $t$ there is a $p'$-element $h$ such that $C_G(h)\neq G$, then $C_G(h)=C_G(z)$ and $C_G(h )=C_G(th)=C_G(t)$. Therefore $C_G(z)=C_G(t)$. Thus, if $C_G(z)$ or $C_G(t)$ contains a non-central $p'$ element, then $C_G(z)=C_G(t)$. Suppose that $|C_G(z)|_{p'}=|C_G(t)|_{p'}=|G/Z|_{p'}$. Then $Ind(G,z)=Ind(G,t)p^n$, in particular $Ind(G,z)$ divides $Ind(G,t)$. Therefore $|C_G(z)|=|C_G(t)|$. 
	
	Thus, we have shown that $SP\subset CH$. An example of a $CH$-group that is not $SP$ is given in Remark \ref{example}.
	
	
	
	
	\section{Proof of Theorem 2}
	It follows from Theorem \ref{primeiro} that any $SP$-group is a $CH$-group. To prove Theorem \ref{segundo}, let us study which $CH$-groups are $SP$-groups.
	
	We will say that $G$ lies in the class $CH(i)$, where $i\in\{I,II,..,VII\}$ and write $G\in CH(i)$ if $G $ is a $CH$ group and item $i$ from Lemma \ref{DJ} is satisfied. Note that the classes $CH(i)$ have intersections, therefore the number $i$ may not be uniquely determined.
	
	Let $G\in SP$, $Z=Z(G)$. 
	\begin{lem}\label{um}
		$G\in F(I)$  if and only if one of the statements is true
		\begin{enumerate}
			\item $G=T\times P$ where $P$ is a Sylow $p$-subgroup of $G$ and $rank(P)=1$; 
			\item $P$ is abelian.
		\end{enumerate}
	\end{lem}
	\begin{proof}
		
		We have that $G$ contains a normal abelian subgroup $A$ of index $p$. Let $H$ be a Hall $p'$-subgroup of $G$. Since $H\leq A$, then $H$ is normal in $G$ and is an abelian group.
		
		
		Let $P\leq H$ be a Sylow $p$-subgroup of $G$. Assume that $P$ acts non-trivially on $H$. From Lemma \ref{Gor_5.2.3} and the fact that $H$ is an Abelian $p'$-group it follows that there is $x\in P$ such that $H=C_H(x)\times Y$, where $ Y>1$. In particular, $Ind(H,x)>1$ and $x\not\in A$. We have $N_G(Y)\geq A\la x\ra$. Since $N_G(Y)/A=G/A$, we have $Y\unlhd G$.
		
		
		Let $y\in Y$. Since $Y$ is a normal subgroup of $G$ and $A\leq C_G(y)$, we have $Ind(G,y)=p$. Note that $Ind(G,x)_p=Ind(P,x)$.
		
		
		Assume that $P$ is not abelian. Assume that $Ind(P,x)=1$. Since $C_G(A)\geq A\la x\ra=P$, we have $A\leq Z(P)$. From the fact that $|P|/|A|=p$ and Lemma \ref{ab} it follows that $P$ is abelian; a contradiction. Therefore, $Ind(G,x)_p>1$. Thus $Ind(G,y)$ divides $Ind(G,x)$; a  contradiction. Therefore $P$ is abelian and Statement $1$ of the Lemma holds.
		
		
		Suppose that $P$ acts trivially on $H$. Therefore, for any $h\in H$ we have $Ind(G,h)_p=1$. It follows from Lemma \ref{Complement}, that $P$ is a direct factor of the group $G$. For any $x\in P$ we have $Ind(G,x)=Ind(P,x)$. Thus, if $P$ is not a group of rank $1$ then $P$ contains elements $x$ and $y$ such that $1<Ind(G,x)<Ind(G,y)$ and $Ind( G,x)|Ind(G,y)$; a contradiction. Therefore $rank(P)=1$.
		
		
		Thus, if $G\in SP$, then one of the statements of the lemma holds.
		
		Let us prove that if one of the statements of the lemma is satisfied, then $G\in SP$.
		
		
		
		Let us assume that the statement $1$ of the Lemma is satisfied. Then $N(G)=\{p\}$ and hence $G\in SP$.
		
		Let us assume that the statement $2$ of the Lemma is satisfied. Then $N(G)=\{p,k\}$, where $k$ is coprime with $p$. Therefore, $G\in SP$.
		
		
		
	\end{proof}
	
	\begin{rem}\label{example}
		Note that the class $F(I)$ contains a $CA$-group that is not an $SP$-group. For example, we can take a $p$-group $P=A\rtimes B$, where $A$ is an elementary abelian group of order $p^p$ and $B$ is a group of order $p$ acting permutatively on $A$. Obviously, the centralizer of any non-central element of $G$ is abelian and $N(G)=\{p,p^{p-1}\}$. If $p>2$, then $G\in CA\setminus SP$.
		This means that the class of $SP$-groups does not coincide with the classes of $CA$- and $CH$-groups.
		
		
	\end{rem}
	
	\begin{lem}\label{dois}
		Groups from $F(II)$ are $SP$-groups.
	\end{lem}
	\begin{proof}
		We have that $G/Z$ is a Frobenius group with Frobenius kernel $K/Z$ and Frobenius
		complement $L/Z$, where $K$ and $L$ are abelian. 
		Let $\overline{\ }: G\rightarrow G/Z$ be a natural homomorphism. Take $x\in K$, such that $Ind(G,x)>1$. The fact that $\overline{C_G(x)}\leq C_{\overline{G}}(\overline{x})$ and $C_{\overline{G}}(\overline{x})=\overline{K}$ implies that $Ind(G,x)=|L/Z|$. Similarly, if $y\in L$ is such that $Ind(G,y)>1$, then $Ind(G,y)=|\overline{K}|$. Since any element of $G$ lies in the subgroup conjugate to one of the subgroups $L$ or $K$, then the set of  conjugacy class sizes of the group $G$ is $\{1,|\overline{K}|, |\overline{L}|\}$. The groups $\overline{K}$ and $\overline{L}$ have coprime order. Thus $G\in SP$.
		
		
	\end{proof}
	\begin{rem}
		Note that groups satisfying statement $2$ of Lemma \ref{um} are $F(II)$ groups.
	\end{rem}

	\begin{lem}
		$G\in F(III)$ if and only if $rank(K)=1$.
		
	\end{lem}
	\begin{proof}
		We have that $G/Z$ is a Frobenius group with Frobenius kernel $K/Z$ and Frobenius
		complement $L/Z$, such that $K=PZ$, where $P$ is a Sylow $p$-subgroup of $G$ for some $p\in \pi(G)$, $P$ is $CH$-group, $Z(P)=Z\cap P$ and $L\simeq HZ$, where $H$ is an abelian $p'$-subgroup of $G$. In particular, $L$ is abelian. 
		
		Suppose that there are elements $x,y\in P$, such that $Ind(P,x)>Ind(P,y)>1$. Since $P$ is a normal Sylow $p$-subgroup of $G$ it follows that $Ind(G,x)_p=Ind(P,x)$. Similar to Lemma~\ref{dois}, it can be shown that $Ind(G,x)=Ind(P,x)|L/Z|$ and $Ind(G,y)=Ind(P,y)|L /Z|$. Thus $Ind(G,x)|Ind(G,y)$; a contradiction. Thus $rank(P)=1$.


		
	\end{proof}
	
	\begin{lem}
		If $G\in F(IV)$, then $G$ is not an $SP$-group.
	\end{lem}
	\begin{proof}
		We have $G/Z\simeq S_4$ and if $V/Z$ is the Klein four group in $G/Z$, then $V$ is nonabelian. 
		
		Let $\overline{G}=G/Z$, $h\in G$ be such that its image $\overline{h}\in \overline{G}$ is of order $4$. Note that $C_{\overline{G}}(\overline{h})=\la h\ra$. From the fact that $\overline{C_{G}(h)}\leq C_{\overline{G}}(\overline{h})$, it follows that $C_G(h)=Z\la h\ra$. Therefore, $Ind(G,h)=6$.
		
		
		Let $g\in G$ be a $2$-element such that $\overline{g}$ is an element of order $2$ and does not lie in any proper normal subgroup of $\overline{G}$. In the standard permutation representation of $S_4$, the element $g$ has the form $(1,2)$. Note that $\overline{C}=C_{\overline{G}}(\overline{g})=\la \overline{g}\ra\times\la \overline{z}\ra$, where $\overline{z}$ is the central element of some Sylow $2$-subgroup of the group $\overline{G}$. In the standard permutation representation of $S_4$, the element $z$ has the form $(1,2)(3,4)$. The group $\overline{C}$ is the Klein four group and therefore $C$ is nonabelian. Note that $Z<C_G(g)\leq C$. Therefore $C_G(g)=Z\la g\ra$. So $Ind(G,g)=12$. Thus, $Ind(G,h)$ divides $Ind(G,g)$; a contradiction.
		
		
	\end{proof}
	\begin{lem}
		$G\in F(V)$ if and only if $rank(P)=1$. 
	\end{lem}
	\begin{proof}
		The statement of the lemma follows from the fact that $N(G)=N(P)$.
		
	\end{proof}      
	\begin{rem}
		Note that the groups from Statement $1$ of Lemma \ref{um} are $F(V)$-groups.
	\end{rem}
	
	\begin{lem}\label{seis}
		Groups from $F(VI)$ are $SP$-groups.
	\end{lem}
	\begin{proof}
		We have $G/Z\simeq PSL(2,p^n)$ or $PGL(2,p^n)$ and $G'\simeq SL(2,p^n)$ where $p$ is a prime and $p^n>3$.
		Put $q=p^n$. It's not difficult to check that $N(SL_2(q))=\{(q^2-1)/2, q(q-1), q(q+1)\}$ and $N(GL_2(q))=\{q(q-1), q^2-1, q(q+1)\}$. 
		
		Let $X\leq G$ be a subgroup of minimal order such that $XZ/Z=G/Z$, $S\leq X$ be such that $S\simeq SL_2(q)$. Since $G=XZ$, we see that $X\unlhd G$ and $N(G)=N(X)$. We have $X=S\rtimes L$, where $L$ is an abelian group acting on $S$ as a group of diagonal automorphisms or $L$ is the trivial group. Thus $N(X)=N(SL_2(q))$ or $GL_2(q)$.
		
	\end{proof}       
	\begin{lem}\label{sete}
		Groups from $F(VII)$ are $SP$-groups.
	\end{lem}
	\begin{proof}
		We have $G/Z\simeq PSL(2,9)$ or $PGL(2,9)$ and $G'$ is isomorphic to the Schur cover of $PSL(2,9)$.
		Let $X\leq G$ be a subgroup of minimal order such that $XZ/Z=G/Z$. Using \cite{GAP} it's easy to get that $N(X)=N(G')=\{72,90,120\}$. Similarly as in Lemma \ref{seis} we obtain that $N(G)=N(S)$ and therefore $G\in SP$.
		
	\end{proof}
	
	\section{Acknowledgements}
	The work was supported by Russian Science Foundation (project № 24-41-10004)

\end{document}